\documentclass[12pt]{article}

\usepackage{a4wide}
\usepackage{graphicx}
\usepackage{amsmath}
\usepackage{amssymb}
\usepackage{amsthm}
\usepackage{enumerate}

\newtheorem{theorem}{Theorem}[section]
\newtheorem{definition}[theorem]{Definition}

\newtheorem{lemma}[theorem]{Lemma}
\newtheorem{claim}[theorem]{Claim}
\newtheorem{corollary}[theorem]{Corollary}
\newtheorem{conjecture}[theorem]{Conjecture}
\newtheorem{construction}[theorem]{Construction}
\usepackage[margin=20pt,small,bf]{caption}

\begin{document}

\title{Partitioning edge-coloured complete graphs into monochromatic cycles and paths}

\author{\large{Alexey Pokrovskiy} \thanks{Research supported by the LSE postgraduate research studentship scheme.} %\footnote{This paper is a draft.  If you want to pass a copy on to someone, please consult the author first.}
\\
\\Departement of Mathematics,
\\ {London School of Economics and Political Sciences,} 
\\ London WC2A 2AE, UK 
\\ {Email: \texttt{a.pokrovskiy@lse.ac.uk}}}

\maketitle

\begin{abstract}
A conjecture of Erd\H{o}s, Gy\'arf\'as, and Pyber says that in any edge-colouring of a~complete graph with $r$ colours, it is possible to cover all the vertices with $r$ vertex-disjoint monochromatic cycles.  So far, this conjecture has been proven only for $r = 2$.  In this paper we show that in fact this conjecture is false for all $r\geq 3$.  In contrast to this, we show that in any edge-colouring of a complete graph with three colours, it is possible to cover all the vertices with three vertex-disjoint monochromatic \emph{paths}, proving a particular case of a conjecture due to Gy\'arf\'as.  As an intermediate result we show that in any edge-colouring of the complete graph with the colours red and blue, it is possible to cover all the vertices with a red path, and a disjoint blue balanced complete bipartite graph.
\end{abstract}
\section{Introduction}
Suppose that the edges of the complete graph on $n$ vertices, $K_n$, are coloured with $r$ colours.  How many vertex-disjoint monochromatic paths are needed to cover all the vertices of $K_n$~?  Gerencs\'er and Gy\'arf\'as~\cite{Gerencser} showed that when $r=2$, this can always be done with at most two monochromatic paths.  For $r>2$, Gy\'arf\'as made the following conjecture.

\begin{conjecture} [Gy\'arf\'as,~\cite{Gyarfas}] \label{Gyarfas}  
The vertices of every $r$-edge coloured complete graph can be covered with $r$ vertex-disjoint monochromatic paths.
\end{conjecture}
Subsequently, Erd\H{o}s, Gy\'arf\'as, and Pyber made the following stronger conjecture.
\begin{conjecture} [Erd\H{o}s, Gy\'arf\'as \& Pyber, \cite{Erdos}] \label{Erdos}
The vertices of every $r$-edge coloured complete graph can be covered with $r$ vertex-disjoint monochromatic cycles.
\end{conjecture}
When dealing with these conjectures, the empty set, a single vertex, and a single edge between two vertices are considered to be paths and cycles. 
It is worth noting that neither of the above conjectures require the monochromatic paths covering $K_n$ to have distinct colours. Whenever a graph $G$ is covered by vertex-disjoint subgraphs $H_1, H_2, \dots, H_k$, we say that $H_1, H_2, \dots, H_k$ \emph{partition} $G$.

  Most effort has focused on Conjecture~\ref{Erdos}.  It was shown in~\cite{Erdos} that there is a function~$f(r)$ such that, for all $n$, any $r$-edge coloured $K_n$ can be partitioned into $f(r)$ monochromatic cycles.
The best known upper bound for $f(r)$ is due to Gy\'arf\'as, Ruszink\'o, S\'ark\"ozy, and Szemer\'edi~\cite{Szemeredi3} who show that, for large $n$, $100r \log{r}$ monochromatic cycles are sufficient to partition the vertices of an $r$-edge coloured $K_n$.

For small $r$, there has been more progress.  The case $r=2$ of Conjecture \ref{Erdos} is closely related to Lehel's Conjecture, which says that any 2-edge coloured complete graph can be partitioned into two monochromatic cycles \emph{with different colours}.  This conjecture first appeared in~\cite{Ayel} where it was proved for some special types of colourings of $K_n$.  Gy\'arf\'as~\cite{Gyarfas2} showed that the vertices of a $2$-edge coloured complete graph can be covered by two monochromatic cycles with different colours intersecting in at most one vertex.  
\L{uczak}, R\"odl, and Szemer\'edi~\cite{Szemeredi2} showed, using the Regularity Lemma, that Lehel's Conjecture holds for $r=2$ for large~$n$.  Later, Allen~\cite{Allen} gave an alternative proof that works for smaller (but still large) $n$, and which avoids the use of the Regularity Lemma.
Lehel's Conjecture was finally shown to be true for all $n$ by Bessy and Thomass\'e~\cite{Thomasse}, using a short, elegant argument.

For $r=3$, Gy\'arf\'as, Ruszink\'o, S\'ark\"ozy, and Szemer\'edi proved the following theorem.
\begin{theorem} [Gy\'arf\'as, Ruszink\'o, S\'ark\"ozy \& Szemer\'edi, ~\cite{Szemeredi1}] \label{threecycles}
  
Suppose that the edges of $K_n$ are coloured with three colours.  There are three vertex-disjoint monochromatic cycles covering all but $o(n)$ vertices in $K_n$.
\end{theorem}
In~\cite{Szemeredi1}, it is also shown that, for large $n$, $17$ monochromatic cycles are sufficient to partition \emph{all} the vertices of every 3-edge coloured $K_n$. 

Despite Theorem~\ref{threecycles} being an approximate version of the case $r=3$ of Conjecture~\ref{Erdos}, the conjecture turns out to be false for all $r\geq 3$.
We prove the following theorem in Section~2.
\begin{theorem}\label{counterexampletheorem}
 Suppose that $r\geq 3$.  There exist infinitely many $r$-edge coloured complete graphs which cannot be vertex-partitioned into $r$ monochromatic cycles.
\end{theorem}

For a particular counterexample of low order to the case $r=3$ of Conjecture \ref{Erdos}, see Figure \ref{figurecounterexample}.  It is worth  noting that in all the $r$-colourings of $K_n$ that we construct, it is possible to cover $n-1$ of the vertices of $K_n$ with $r$ disjoint monochromatic cycles.  Therefore the counterexamples we construct are quite ``mild'' and leave room for further work to either find better counterexamples, or to prove approximate versions of the conjecture similar to Theorem~\ref{threecycles}.

\begin{figure}
  \centering
    \includegraphics[width=0.6\textwidth]{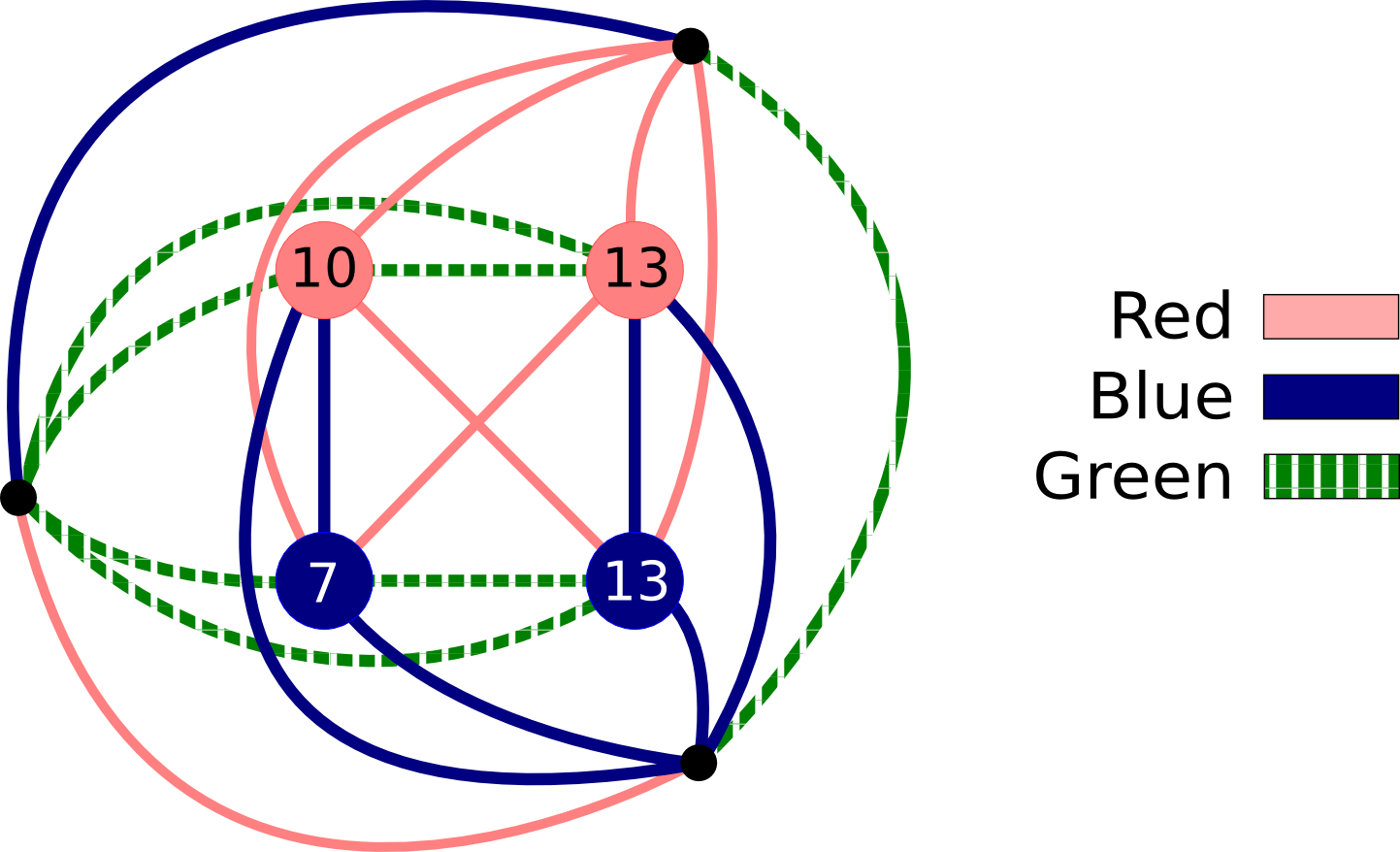}
  \caption{A $3$-edge colouring of $K_{46}$ which cannot be partitioned into three monochromatic cycles.  The small black dots represent single vertices.  The large red and blue circles represent red and blue complete graphs of order specified by the numbers inside.  The coloured lines between the sets represent all the edges between them being of that colour.
  This particular colouring is called $J_3^1$ in this paper.  In Section~2 we prove that this colouring does not allow a partition into three monochromatic cycles.} \label{figurecounterexample}
\end{figure}

Theorem \ref{counterexampletheorem} also raises the question of whether Conjecture~\ref{Gyarfas} holds for $r\geq 3$ or not.  The second main result of this paper is to prove the case $r=3$ of Conjecture~\ref{Gyarfas}.

\begin{theorem} \label{threepaths}
For $n\geq 1$, suppose that the edges of $K_n$ are coloured with three colours.  There is a~vertex-partition of $K_n$ into three monochromatic paths.
\end{theorem}

Theorem~\ref{threepaths} is proved in Section~3. % It is worth noting that our method works for all $n$, and does not use the Regularity Lemma.  

One way of generalizing the conjectures and theorems above is to consider partitions of an $r$-edge coloured graph $G$ other than the complete graph.  Some results in this direction were already obtained in~\cite{Sarkozy} where $G$ is an arbitary graph with specified independence number, and in~\cite{Balogh} where $G$ is an arbitary graph with $\delta(G)\geq \frac{3}{4}$.  In order to prove Theorem~\ref{threepaths} we will consider partitions of a 2-edge coloured balanced complete bipartite graph (the complete bipartite graph $K_{n,m}$ is called \emph{balanced} if $n=m$ holds).  In order to state our result we will need the following definition.

\begin{definition}
Let $K_{n,n}$ be a 2-edge coloured balanced complete bipartite graph with partition classes $X$ and $Y$.  We say that the colouring on $K_{n,n}$ is \textbf{split} if it is possible to partition $X$ into two nonempty sets $X_1$ and $X_2$, and $Y$ into two nonempty sets $Y_1$ and $Y_2$, such that the following hold.    
\begin{itemize}
\item The edges between $X_1$ and $Y_2$, and the edges between $X_2$ and $Y_1$  are {red}.  
\item The edges between $X_1$ and $Y_1$, and the edges between $X_2$ and $Y_2$  are {blue}.  
\end{itemize}
The sets $X_1$, $X_2$, $Y_1$, and $Y_2$ will be called the ``classes" of the split colouring.
\end{definition}
When dealing with split colourings of $K_{n,n}$ the classes will always be labeled ``$X_1$'', ``$X_2$'', ``$Y_1$'', and  ``$Y_2$'' with colours between the classes as in the above definition.  See Figure~\ref{figuresplit} for an illustation of a split colouring of $K_n$.

\begin{figure}
  \centering
    \includegraphics[width=0.4\textwidth]{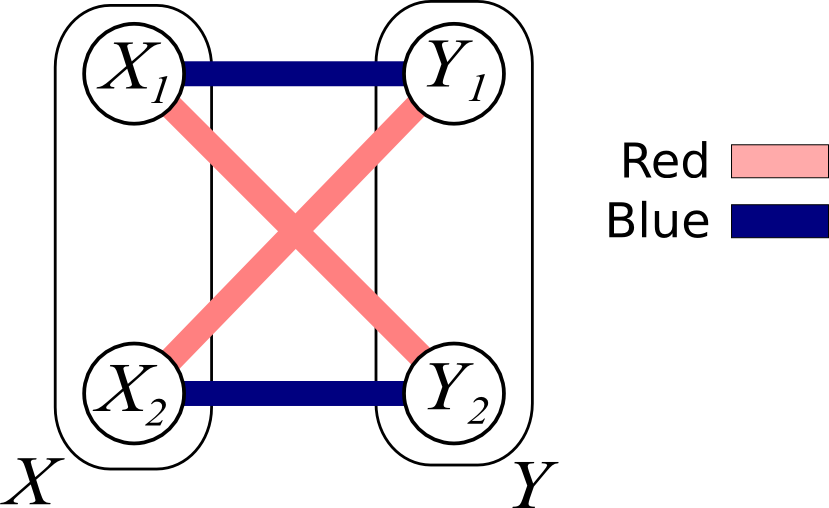}
  \caption{A split colouring of $K_{n,n}$.} \label{figuresplit}
\end{figure}

These colourings have previously appeared in the following theorem due to Gy\'{a}rf\'{a}s and Lehel.  The proof of this theorem appears implicitly in~\cite{Lehel}, and the statement appears in~\cite{Gyarfas2}.
\begin{theorem}  [Gy\'{a}rf\'{a}s \& Lehel,~\cite{Gyarfas2, Lehel}]   \label{GyarfasLehel}
Suppose that the edges of $K_{n,n}$ are coloured with two colours.  If the colouring is not split, then there exist two disjoint monochromatic paths with different colours which cover all, except possibly one, of the vertices of $K_{n,n}$.
\end{theorem}
We will prove the following slight extension of Theorem~\ref{GyarfasLehel}.
\begin{theorem}\label{bipartitepartition}
Suppose that the edges of $K_{n,n}$ are coloured with two colours.  There is a vertex-partition of $K_{n,n}$ into two monochromatic paths with different colours if, and only if, the colouring on $K_{n,n}$ is not split.
\end{theorem}

There exist split colourings of $K_{n,n}$ which cannot be partitioned into two monochromatic paths  even when we are allowed to repeat colours.  Indeed, any split colouring with classes $X_1$, $X_2$, $Y_1$, and $Y_2$, satisfying $\big||X_1|-|Y_1|\big|\geq 2$ and $\big||X_1|-|Y_2|\big|\geq 2$ will have this property.  Using Theorem~\ref{bipartitepartition}, it is not hard to show that any 2-colouring of $K_{n,n}$ which cannot be partitioned into two monochromatic paths must be a split colouring with class sizes as above.

However, it is easy to check that every 2-edge coloured $K_{n,n}$ which is split can be partitioned into three monochromatic paths, so the following corollary follows from either of the above two theorems.

\begin{corollary}\label{bipartitepartitionsimple}
Suppose that the edges of $K_{n,n}$ are coloured with two colours.  There is a~vertex-partition of $K_{n,n}$ into three monochromatic paths.
\end{corollary}

Recall that Gerencs\'er and Gy\'arf\'as showed that any 2-edge coloured complete graph $K_n$ can be partitioned into two monochromatic paths.  The following lemma, which may be of independent interest, shows that one of the paths partitioning $K_n$ can be replaced by a graph which has more structure.

\begin{lemma}\label{path-bipartitecoloured}
 Suppose that the edges of $K_n$ are coloured with the colours {red} and {blue}.  There is a vertex-partition of $K_n$ into a {red} path and a {blue} balanced complete bipartite graph.
\end{lemma}

Lemma~\ref{path-bipartitecoloured} and Corollary~\ref{bipartitepartitionsimple} together easily imply that a 3-edge coloured complete graph can be partitioned into \emph{four} monochromatic paths.  Indeed suppose that $K_n$ is coloured with the colours red, blue and green.  First we treat {blue} and {green} as a single colour and apply Lemma~\ref{path-bipartitecoloured} to obtain a partition of $K_n$ into a {red} path and a {blue}-{green} balanced complete bipartite graph.  Now apply Corollary~\ref{bipartitepartitionsimple} to this graph to obtain a partition of $K_{n}$ into four monochromatic paths.

The proof of Theorem~\ref{threepaths} is more involved, and we will need a more refined version of Lemma~\ref{path-bipartitecoloured} which is stated and proved in Section~3.1.

\section{Counterexamples to the conjecture of Erd\H{o}s, Gy\'arf\'as, and  Pyber.}

In this section, for $r\geq 3$, we will prove Theorem~\ref{counterexampletheorem}, by constructing a sequence of $r$-edge coloured complete graphs, $J_r ^m$, which cannot be partitioned into $r$ monochromatic cycles.  First we will construct a sequence of $r$-edge coloured complete graphs, $H_r ^m$, which cannot be partitioned into $r$ monochromatic paths \emph{with different colours}.

\begin{construction}\label{affineblowup}
Let $r$ and $m$ be integers such that $r\geq 3$ and $m \geq 1$.
We construct $r$-edge coloured complete graphs, $H_r ^m$, on $\left( \frac{87}{250}5^r-\frac{1}{2}\right)m$ vertices.  First, we will consider the $r=3$ and $m\geq 1$ cases and define the colourings $H_3^m$. Then for $r\geq 4$ we will then construct the graphs $H_r^m$ recursively out of $H_{r-1} ^{5m}$ .
\begin{itemize}
\item Suppose that $r=3$.  To construct $H_3^m$, we 3-edge colour $K_{43m}$ as follows.  We partition the vertex set of $K_{43m}$ into four classes $A_1$, $A_2$, $A_3$, and $A_4$ such that $|A_1|=10m$, $|A_2|=13m$, $|A_3|=7m$, and $|A_4|=13m$.  The edges between $A_1$ and $A_2$ and between $A_3$ and $A_4$ are colour $1$.  The edges between $A_1$ and $A_3$ and between $A_2$ and $A_4$ are colour $2$.  The edges between $A_1$ and $A_4$ and between $A_2$ and $A_3$ are colour $3$.  The edges within $A_1$ and $A_2$ are colour $3$.  The edges within $A_3$ and $A_4$ are colour $2$.

\item Suppose that $r\geq 4$. Note that the $|H_r^m|=|H_{r-1} ^{5m}|+2m$ holds, so we can partition the vertices of $H_r^m$ into two sets $H$ and $K$ such that $|H|=|H_{r-1} ^{5m}|$ and $|K|=2m$. We colour $H$ with colours $1,\dots, r-1$ to produce a copy of $H_{r-1} ^{5m}$.  All other edges are coloured with colour $r$.
\end{itemize}
\end{construction}
%As we shall see, $H_r^m$ satisfies Lemma~\ref{blowuplemma} for any $m$.  We are now ready to construct graphs $J_r^m$ which satisfy Theorem~\ref{counterexampletheorem}.

\begin{construction}\label{counterexample}
Let $r$ and $m$ be integers such that $r\geq 3$ and $m \geq 1$.
We construct $r$-edge coloured complete graphs, $J_r ^m$, on $\left( \frac{87}{250}5^r-\frac{1}{2}\right)m +r$ vertices.

We partition the vertices of $|J_r ^m|$ into a set $H$ of order $|H_r ^m|$ and a set of $r$ vertices $\{v_1, \dots, v_r\}$.  The edges in $H$ are coloured to produce a copy of $H_r ^m$.  For each $i \in \{1, \dots, r\}$, we colour all the edges between $v_i$ and $H$ with colour $i$.  The edge $v_1 v_2$ is colour $3$.  For $j\geq 3$ the edge $v_1v_j$ is colour $2$ and the edge $v_2v_j$ is colour $1$.  For $3\leq i < j$, the edge $v_i v_j$ is colour $1$.
\end{construction}

The following simple fact will be convenient to state.
\begin{lemma}\label{independentset}
Let $G$ be a graph, $X$ an independent set in $G$, and $P$ a path in $G$.  Then we have 
$$|P\cap X|\leq |P \cap (G\setminus X)|+1.$$
\end{lemma}
\begin{proof}
Let $x_1\dots x_k$ be the vertex sequence of $P$.  For $i \leq k-1$, if $x_i$ is in $X$, then $x_{i+1}$ must be in $G\setminus X$, implying the result.
\end{proof}

The only property of the graphs $H_r ^m$ that we will need is that they satisfy the following lemma.
\begin{lemma}\label{blowuplemma}
Let $r\geq 3$ and $m \geq 1$ be integers.  The following both hold.
\begin{enumerate} [(i)]
\item $H_r ^m$ cannot be vertex-partitioned into $r-1$ monochromatic paths.
\item $H_r ^m$ cannot be vertex-partitioned into $r$ monochromatic paths with different colours.
\end{enumerate}
\end{lemma}

\begin{proof}
It will be convenient to prove a slight strengthening of the lemma.  Let $T$ be any set of at most $m$ vertices of $H_r ^m$.  We will prove that the graph $H_r ^m\setminus T$ satisfies parts (i) and (ii) of the lemma.

The proof is by induction on $r$.  First we shall prove the lemma for the initial case, $r=3$.

Recall that $H_3^m$ is partitioned into four sets $A_1$, $A_2$, $A_3$, and $A_4$.  Let $B_i=A_i\setminus T$.  Since $|T|\leq m$, the sets $B_1$, $B_2$, $B_3$, and $B_4$ are all nonempty. We will need the following claim. 

\begin{claim} \label{singlepathclaim}
The following hold.
\begin{enumerate}[(a)]
\item $B_2$ cannot be covered by a  colour $1$ path.
\item $B_1$ cannot be covered by a  colour $2$ path.
\item $B_4$ cannot be covered by a  colour $3$ path.
\item $B_4$ cannot be covered by a  colour $1$ path.
\item $B_1 \cup B_3$ cannot be covered by a colour $1$ path contained in  $B_1\cup B_2$ and a disjoint colour $3$ path contained in $B_2\cup B_3$.
\item $B_2 \cup B_3$ cannot be covered by a colour $1$ path contained in  $B_3\cup B_4$ and a disjoint colour $2$ path contained in $B_2\cup B_4$.
\end{enumerate}
\end{claim}
\begin{proof}
\
\begin{enumerate}[(a)]

\item Let $P$ be any colour $1$ path in $H_3 ^m \setminus T$ which intersects $B_2$.  The path $P$ must then be contained in the colour $1$ component $B_1 \cup B_2$.  The set $B_2$ does not contain any colour $1$ edges, so Lemma~\ref{independentset} implies that $|P\cap  B_2|\leq |P\cap B_1| + 1$ holds.  This, combined with the fact that $|T|\leq m$ holds, implies that we have
\begin{equation*} \
  |P\cap  B_2|\leq |P\cap B_1| + 1 \leq |A_1| + 1=10m + 1 < 12m \leq |B_2|.
\end{equation*}
This implies that $P$ cannot cover all of $B_2$.

\item This part is proved similarly to (a), using the fact that $B_1$ does not contain any colour~$2$ edges and that we have $|A_3|+1=7m+1<9m\leq |B_1|$.

\item This part is proved similarly to (a), using the fact that $B_4$ does not contain any colour~$3$ edges and that we have $|A_1|+1=10m+1<12m\leq |B_4|$.

\item This part is proved similarly to (a), using the fact that $B_4$ does not contain any colour~$1$ edges and that we have $|A_3|+1=7m+1<12m\leq |B_4|$.

\item Let $P$ be a colour $1$ path contained in  $B_1\cup B_2$ and let $Q$ be a disjoint colour $3$ path contained in $B_2\cup B_3$.  The set $B_1$ does not contain any colour $1$ edges and $B_3$ does not contain any colour $3$ edges, so Lemma~\ref{independentset} implies that $|(P\cup Q) \cap  (B_1\cup B_3)|\leq |(P\cup Q) \cap B_2| + 2$ holds.  This, combined with the fact that $|T|\leq m$ holds, implies that we have
\begin{equation*} 
|(P\cup Q) \cap  (B_1\cup B_3)| \leq |(P\cup Q) \cap B_2| + 2
\leq |A_2|+2 = 13m+2 < 16m \leq  (B_1\cup B_3)
 \end{equation*}
This implies that $P$ and $Q$ cannot cover all of $B_1\cup B_3$.

\item This part is proved similarly to (e), using the fact that $B_2$ does not contain any colour~$2$ edges, $B_3$ does not contain any colour $1$ edges, and that we have  $|A_4|+2=13m+2<19m\leq B_2\cup B_3$.

\end{enumerate}

\end{proof}
We now prove the lemma for $r=3$.  We deal with parts (i) and (ii) separately
\begin{enumerate} [(i)]
\item  Suppose, for the sake of contradiction, that $P$ and $Q$ are two monochromatic paths which partition $H_3 ^m \setminus T$.  Note that $P$ and $Q$ cannot have different colours since any two monochromatic paths with different colours in $H_3 ^m$ can intersect at most three of the four sets $B_1$, $B_2$, $B_3$, and $B_4$.
The colouring $H_3 ^m \setminus T$ has exactly two components of each colour, so, for each $i$, the set $B_i$ must be covered by either $P$ or $Q$.  This contradicts case (a), (b), or (c) of Claim \ref{singlepathclaim} depending on whether $P$ and $Q$ hove colour $1$, $2$, or $3$.

\item  Suppose, for the sake of contradiction, that $P_1$, $P_2$, and $P_3$ are three monochromatic paths which partition $H_3 ^m \setminus T$ such that $P_i$ has colour $i$.  

Suppose that $P_2 \subseteq B_1\cup B_3$. By parts (c) and (d) of Claim \ref{singlepathclaim},  both of the paths $P_1$ and $P_3$ must intersect $B_4$.  This leads to a contradiction since none of the paths $P_1$, $P_2$, and $P_3$ intersect $B_2$.

Suppose that $P_2 \subseteq B_2\cup B_4$.
If $P_1\subseteq B_1\cup B_2$ then $P_3$ must be contained in $B_2 \cup B_3$, contradicting part (e) of Claim \ref{singlepathclaim}.
If $P_1\subseteq B_3\cup B_4$ then $P_3$ must be contained in $B_1\cup B_4$, contradicting part (f) of Claim \ref{singlepathclaim}.
This completes the proof of the lemma for the case $r=3$.
\end{enumerate}

We now prove  the lemma  for  $r\geq 3$ by induction on $r$.  The initial case $r=3$ was proved above.
Assume that the lemma holds for $H_{r-1} ^m$, for all $m \geq 1$.  Let $H$ and $K$ partition~$H_{r} ^m$ as in the definition of $H_{r} ^m$.  Suppose that $H_r ^m \setminus T$ is partitioned into $r$ monochromatic paths $P_1, \dots, P_r$ (with possibly some of these empty).  Without loss of generality we may assume that these are ordered such that each of the paths $P_1,\dots, P_k$ intersects $K$, and that each of the paths $P_{k+1},\dots, P_r$ is disjoint from $K$. Note that we have $k\leq |K|= 2m$.  Let $S= H\cap (P_1  \cup \dots \cup P_k)$.   The set $H\setminus T$ does not contain any colour $r$ edges, so Lemma~\ref{independentset} implies that we have $|S|\leq |K|+k \leq 4m$, and so $|S\cup T|\leq 5m$.  We know that $H \setminus (S\cup T)$ is partitioned into $r-k$ monochromatic paths $P_{k+1},\dots, P_r$, so, by induction, we know that $k=1$ and that the paths $P_{2},\dots, P_r$ are all nonempty and do not all have different colours.  This completes the proof since we know that $P_1$ contains vertices in $K$, and hence $P_1, \dots, P_r$ are all nonempty, and do not all have different colours.

\end{proof}

We now prove the main result of this section.

\begin{theorem}
For $r\geq 3$ and $m\geq 1$, $J_r ^m$ cannot be vertex-partitioned into $r$ monochromatic cycles.
\end{theorem}

\begin{proof}
Let $H$ and $\{v_1, \dots, v_r\}$ partition $J_r ^m$ as in the definition of $J_r ^m$
Suppose that $C_1, \dots, C_r$ are $r$ disjoint monochromatic cycles in $J_r ^m$.  We need to show that $C_1\cup \dots \cup C_r \neq J_r^m$.  Note that, for any $i\neq j$, the edge $v_iv_j$ has a different colour to the edges between $v_i$ and $H$.  This means that monochromatic cycle in $J_r ^m$ cannot simultaneously pass through edges in $\{v_1, \dots, v_r\}$ and vertices in $H$.

Let $P_i=C_i \setminus \{v_1,\dots, v_r\}$.  We claim that, for each $i$, $P_i$ is a monochromatic path in $H$.  If $C_i \cap \{v_1, \dots, v_r\} \leq 1$, then this is clear.  So, suppose that for $j\neq k$ we have $v_j, v_k \in C_i$.  In this case $C_i$ cannot contain vertices in $H$, since otherwise the edges of $C_i$ which pass through $v_j$ and $v_k$ would have different colours, contradicting the fact that $C_i$ is monochromatic.  This means that $P_i = \emptyset$, which is trivially a path.

Therefore $P_1, \dots, P_r$ partition $H$ into $r$ monochromatic paths.  By Lemma~\ref{blowuplemma}, they are all nonempty and not all of different colours.  This means that there is a colour, say colour $i$, which is not present in any of the cycles $C_1, \dots, C_r$.  For each $j$, the fact that $P_j$ is nonempty implies that $C_j$ does not contain edges in $\{v_1, \dots, v_r\}$.  But then, the vertex $v_i$ cannot be contained in any of the cycles $C_1, \dots, C_r$ since all the edges between $v_i$ and $H$ have colour $i$.
\end{proof}

\section{Partitioning a 3-coloured complete graph into three monochromatic paths.}
In this section we prove Theorem~\ref{threepaths}.

Throughout this section, when dealing with $K_{n,n}$, the classes of the bipartition will always be called $X$ and $Y$.  
For two sets of vertices $S$ and $T$ in a graph $G$, let $B(S,T)$ be the subgraph of $G$ with vertex set $S\cup T$ with $st$ an edge of $B(S,T)$ whenever $s\in S$ and $t \in T$.  
A \emph{linear forest} is a disjoint union of paths.

For a nonempty path $P$, it will be convenient to distinguish between the two endpoints of $P$ saying that one endpoint is the ``start'' of $P$ and the other is the ``end'' of $P$.  Thus we will often say things like ``Let $P$ be a path from $u$ to $v$''.  Let $P$ be a path from $a$ to $b$ in $G$ and $Q$ a path from $c$ to $d$ in $G$.  If $P$ and $Q$ are disjoint and $bc$ is an edge in $G$, then we define $P+Q$ to be the unique path from $a$ to $d$ formed by joining $P$ and $Q$ with the edge $bc$.    If $P$ is a path and $Q$ is a subpath of $P$ sharing an endpoint with $P$, then $P-Q$ will denote the subpath of $P$ with vertex set $V(P)\setminus V(Q)$.

We will often identify a graph $G$ with its vertex set $V(G)$.  Whenever we say that two subgraphs of a graph are ``disjoint" we will always mean vertex-disjoint.  If $H$ and $K$ are subgraphs of $G$ then $H\setminus K$ will mean $V(H)\setminus V(K)$ and $H\cup K$ will mean $V(H)\cup V(K)$.  Additive notation will be reserved solely for concatenating paths as explained above.

All colourings in this section will be edge-colourings.
Whenever a graph is coloured with two colours, the colours will be called ``{red}" and ``{blue}".  If there are three colours, they will be ``{red}", ``{blue}", and ``{green}".  
If a graph $G$ is coloured with some number of colours we define the \emph{{red} colour class} of $G$ to be the subgraph of $G$ with vertex set $V(G)$ and edge set consisting of all the {red} edges of $G$.  
We say that $G$ is \emph{connected in} {red}, if the {red} colour class is a connected graph.  Similar definitions are made for {blue} and {green} as well.

We will need the following special 3-colourings of the complete graph.

\begin{definition}\label{4partitedefinition}
 Suppose that the edges of $K_n$ are coloured with three colours.  We say that the colouring is \textbf{4-partite} if there exists a partition of the vertex set into four nonempty sets $A_1$, $A_2$, $A_3$, and $A_4$ such that the following hold.
\begin{itemize}
 \item The edges between $A_1$ and $A_4$, and the edges between $A_2$ and $A_3$  are {red}.
 \item The edges between $A_2$ and $A_4$, and the edges  between  $A_1$ and $A_3$ are {blue}. 
 \item The edges between $A_3$ and $A_4$, and the edges  between $A_1$ and $A_2$ are {green}. 

\end{itemize}
 The edges within the sets $A_1$, $A_2$, $A_3$, and $A_4$ can be coloured arbitrarily.  The sets $A_1$, $A_2$, $A_3$, and $A_4$ will be called the ``classes" of the $4$-partition. 
\end{definition}

When dealing with 4-partite colourings of $K_n$, the classes will always be labeled ``$A_1$'', ``$A_2$'', ``$A_3$'', and  ``$A_4$'', with colours between the classes as in the above definition.  See Figure~\ref{figure4partite} for an illustation of a 4-partite colouring of $K_n$

\begin{figure}
  \centering
    \includegraphics[width=0.4\textwidth]{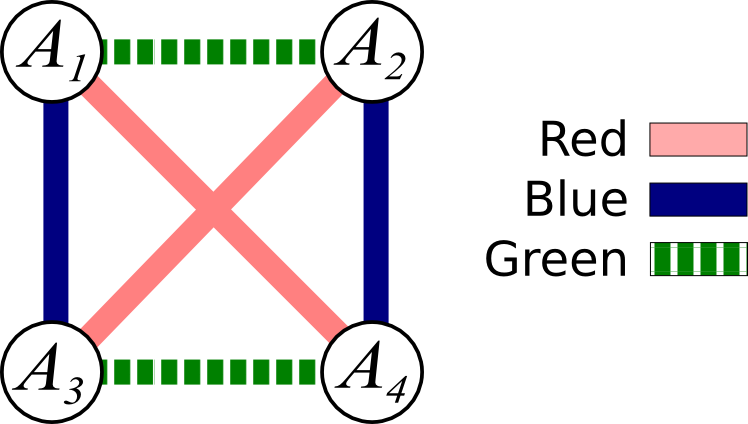}
  \caption{A 4-partite colouring of $K_n$.} \label{figure4partite}
\end{figure}

For all other notation, we refer to~\cite{Diestel}.

In Section~2, we saw that there exist 3-colourings of the complete graphs which cannot be partitioned into three monochromatic paths \emph{with different colours}.  It turns out that all such colourings must be 4-partite.  Our proof of Theorem~\ref{threepaths} will split into the following two parts.

\begin{theorem} \label{not4partitecase}
Suppose that the edges of $K_n$ are coloured with three colours such that the colouring is not 4-partite.  Then $K_n$ can be vertex-partitioned into three monochromatic paths with different colours.
\end{theorem}

\begin{theorem} \label{4partitecase}
 Suppose that the edges of $K_n$ are coloured with three colours such that the colouring is 4-partite.  Then $K_n$ can be vertex-partitioned into three monochromatic paths, at most two of which have the same colour.
\end{theorem}

We will use Theorem \ref{not4partitecase} in the proof of Theorem \ref{4partitecase}.

\subsection{Proof of Theorem~\ref{not4partitecase}.}

We begin by proving the following strengthening of Lemma~\ref{path-bipartitecoloured}
\begin{lemma}\label{path-bipartite}
Let $G$ be a graph, and $v$ a vertex in the largest connected component of $G$.  There is a vertex-partition of $G$ into a path $P$, and two sets $A$ and $B$, such that there are no edges between $A$ and $B$, and $|A|=|B|$.  In addition $P$ is either empty or starts at $v$.
\end{lemma}
\begin{proof}
Let $C$ be the largest connected component of $G$.
We claim that there is a partition of $G$ into a path $P$ and two sets $A$ and $B$ such that the following hold:
\begin{enumerate} [(i)]
\item $|A|\leq |B|$.
\item There are no edges between $A$ and $B$.
\item $P$ is either empty or starts from $v$.
\item $|A\setminus C|$ is as large as possible (whilst keeping (i) - (iii) true).
\item $|A|$ is as large as possible (whilst keeping (i) - (iv) true).
\item $|P|$ is as large as possible (whilst keeping (i) - (v) true).
\end{enumerate}

To see that such a partition exists, note that letting $P=A=\emptyset$ and $B=V(G)$ gives a partition satisfying (i) - (iii), so there must be a partition having $|A\setminus C|$, $|A|$, and $|P|$ maximum, as required by (iv) - (vi).  

Assume that $P$, $A$ and $B$ satisfy (i) - (vi).  We claim that $|A|=|B|$ holds.  
Suppose, for the sake of contradiction, that we have $|A| < |B|$.  

Suppose that $P$ is empty.  There are two cases depending on whether $C\subseteq A$ or $C\subseteq B$ holds.  Note that, by (ii), we are always in one of these cases.
\begin{itemize}
 \item Suppose that $C\subseteq A$.  By (i) and (ii), there must be some connected component of $G$, say $D$, which is contained in $B$.  In this case, let $P'=P$, $A'=(A\setminus C) \cup D$, and $B'=(B\setminus D) \cup C$.  Using $|D|\leq |C|$ we obtain that $|A'|\leq |B'|$ holds.  Therefore $P'$, $A'$, and $B'$ partition $G$, satisfy (i) - (iii), and have $|A'\setminus C|=|A|-|C|+|D|> |A|-|C|=|A\setminus C|$.  This contradicts $|A\setminus C|$ being maximal in the original partition.
 \item Suppose that $C\subseteq B$.  In this case we have $v\in B$. Letting $P'=\{v\}$, $A'=A$, and $B'=B\setminus \{v\}$ gives a partition satisfying (i) - (v), and having $|P'| > |P|$.  This contradicts $P$ being maximal in the original partition.
\end{itemize}

Suppose that $P$ is not empty.  Let $u$ be the end vertex of $P$.  There are two cases depending on whether there are any edges between $u$ and $B$

\begin{itemize}
 \item Suppose that for some $w\in B$, $uw$ is an edge.  Letting $P'=P+w$, $A'=A$, and $B'=B\setminus \{w\}$ gives  a partition satisfying (i) - (v), and having $|P'| > |P|$.  This contradicts $P$ being maximal in the original partition.
 \item Suppose that for all $w\in B$, $uw$ is not an edge.  Letting $P'=P-u$, $A'=A\cup \{u\}$, and $B'=B$ gives partition satisfying (i) - (iv), and having $|A'|>|A|$.  This contradicts $A$ being maximal in the original partition.
\end{itemize}
\end{proof}

 Lemma~\ref{path-bipartite} implies Lemma~\ref{path-bipartitecoloured}, by taking $G$ to be the {red} colour class of a $2$-coloured complete graph.  
The following could be seen as a strengthening of Lemma~\ref{path-bipartitecoloured}, when one of the colour classes of $K_n$ is connected.

\begin{lemma}\label{path-bipartiteconnected}
 Suppose that $G$ is connected graph.  Then at least one of the following holds.
\begin {enumerate}[(i)]
\item There is a path $P$ passing through all, but one vertex in $G$. %There is a partition of $G$ into a path $P$ and a single vertex $v$.
\item There is a vertex-partition of $G$ into a path $P$, and three nonempty sets $A$, $B_1$, and $B_2$ such that $|A|=|B_1|+|B_2|$ and there are no edges between any two of $A$, $B_1$, and $B_2$.
\end {enumerate}
\end{lemma}
\begin{proof}
First suppose that for every path $P$, $G\setminus P$ is connected.  Let $P$ be a path in $G$ of maximum length.  Let $v$ be an endpoint of $P$.  By maximality, $v$ cannot be connected to anything in $(G\setminus P) \cup \{v\}$.  However, since $P - \{v\}$ is a path, $(G\setminus P) \cup \{v\}$ must be connected, hence it consists of the single vertex $v$.  Thus the path $P$ passes through every vertex in $G$, proving case (i) of the lemma.

Now, we can assume that there exists a path $P_0$ such that $G\setminus P_0$ is disconnected.  In addition, we assume that $P_0$ is a shortest such path.  The assumption that $G$ is connected implies that $P_0$ is not empty.  Suppose that $P_0$ starts with $v_1$ and ends with $v_2$.  Let $C_1, \dots, C_j$ be the connected components of $G\setminus P_0$, ordered such that $|C_1|\geq |C_2|\geq \dots \geq |C_j|$.    The assumption of $P_0$ being a shortest path, such that $G\setminus P_0$ is disconnected, implies that $v_1$ and $v_2$ are both connected to  $C_t$ for each $t \in \{1,\dots,j\}$.  Indeed if this were not the case, then either $P_0 - \{v_1\}$ or $P_0 - \{v_2\}$ would give a shorter path with the required property.

Let $u_1$ be a neighbour of $v_1$ in $C_{1}$ and $u_2$ a neighbour of $v_2$ in $C_2$.  Apply Lemma~\ref{path-bipartite} to~$C_{1}$ to obtain a partition of $C_1$ into a path $P_1$ and two sets $X_1$ and $Y_1$, such that $|X_1|=|Y_1|$ and there are no edges between $X_1$ and $Y_1$.  Similarly, apply Lemma~\ref{path-bipartite} to $C_2 \cup \dots \cup C_{j}$ to obtain a partition of $C_2 \cup \dots \cup C_{j}$ into a path $P_2$ and two sets $X_2$ and $Y_2$, such that $|X_2|=|Y_2|$ and there are no edges between $X_2$ and $Y_2$.    In addition we can assume that $P_1$ is either empty or ends at $u_1$ and that $P_2$ is either empty or starts at $u_2$.   Since $v_1u_1$ and $v_2u_2$ are both edges, we can define a path  $Q=P_1 + P_0 + P_2$.  Let $w_1$ be the start of $Q$, and $w_2$ the end of $Q$.  We have that either $w_1 \in C_1$ or $w_1 =v_1$ and either $w_2\in C_2\cup \dots\cup C_j$ or $w_2=v_2$.

If each of the sets $X_1$, $Y_1$, $X_2$, and $Y_2$ is nonempty, then case (ii) of the lemma holds, using the path $Q$, $A=X_1 \cup X_2$, $B_1=Y_1$, and $B_2=Y_2$.

Suppose that $X_1=Y_1=\emptyset$ and  $X_2=Y_2\neq \emptyset$.  In this case $w_1$ must lie in $C_1$ since we know that $P_1 \cup X_1\cup X_2= C_1 \neq \emptyset$.  Therefore $P_1$ is nonempty, and so must contain $w_1$.

Suppose that $w_2$ has no neighbours in $X_2 \cup Y_2$.  Note that in this case $w_2\neq v_2$ since otherwise $X_2\cup Y_2= C_2\cup\dots\cup C_j$ would hold, and we know that $v_2$ has neighbours in $C_2\dots C_j$.  Therefore, we have $w_2\in C_2\cup\dots\cup C_j$ , and so (ii) holds with $P=Q-\{w_1\} -\{w_2\}$ as our path, 
 $A=X_2\cup\{w_2\}$, $B_1=Y_2$, and $B_2=\{w_1\}$.  

Suppose that $w_2$ has a neighbour $x$ in $X_2\cup Y_2$.  Without loss of generality, assume that $x \in X_2$.  If $|X_2|=|Y_2|= 1$, then case (i) of the lemma holds with  $Q+x$ a path covering all the vertices in $G$ except the single vertex in $Y_2$.  If $|X_2|=|Y_2|\geq 2$ then case (ii) holds with $P=Q+\{x\}-\{w_1\}$ as our path, $A=Y_2$, $B_1=X_2 - \{x\}$, and $B_2=\{w_1\}$.  

The case when $X_1=Y_1\neq\emptyset$ and $X_2=Y_2=\emptyset$ is dealt with similarly.  If $X_1=Y_1=X_2=Y_2=\emptyset$, then $Q$ covers all the vertices in $G$, so case (i) holds.
\end{proof}

The following lemma gives a characterization of split colourings of $K_{n,n}$.

\begin{lemma} \label{split}
Let $K_{n,n}$ be a 2-edge coloured balanced complete bipartite graph.  The colouring on $K_{n,n}$ is split if and only if none of the following hold.
\begin{enumerate} [(i)]
 \item $K_{n,n}$ is connected in some colour.
 \item There is a vertex $u$ such that all the edges through $u$ are the same colour.
 \end{enumerate}
\end{lemma}
\begin{proof}
 Suppose that $K_{n,n}$ is not split and (i) fails to hold.  We will show that (ii) holds.  Let $X$ and $Y$ be the classes of the bipartition of $K_{n,n}$.  Let $C$ be any {red} component of $K_{n,n}$, $X_1=X\cap C$, $X_2=X \setminus C$, $Y_1=Y \cap C$, and $Y_2= Y\setminus C$.  If all these sets are nonempty, then $G$ is split with classes $X_1$, $X_2$, $Y_1$, and $Y_2$.  To see this note that there cannot be any {red} edges between $X_1$ and $Y_2$, or between $X_2$ and $Y_1$ since $C$ is a {red} component.  There cannot be any {blue} edges between $X_1$ and $Y_1$, or between $X_2$ and $Y_2$ since $K_{n,n}$ is disconnected in {blue}.  

Assume that one of the sets $X_1$, $X_2$, $Y_1$, or $Y_2$ is empty.
If $X_1$ is empty, then $C$ is entirely contained in $Y$ and hence consists of a single vertex $u$, giving rise to case (ii) of the lemma. 
If $X_2$ is empty, then note that $Y_2$ is nonempty, since otherwise $C=K_{n,n}$ would hold contradicting our assumption that (i) fails to hold.  Let $u$ be any vertex in $Y_2$.  For any $v$, the edge $uv$ must be {blue}, since $X\subseteq C$ holds. Thus again (ii) holds.  The cases when $Y_1$ or $Y_2$ are empty are done in the same way by symmetry.

For the converse, note that if $K_{n,n}$ is split, then the {red} components are $X_1\cup Y_1$ and $X_2\cup Y_2$, and that the {blue} components are $X_1\cup Y_2$ and $X_2 \cup Y_1$.  It is clear that neither (i) nor (ii) can hold.
\end{proof}

We now prove Theorem~\ref{bipartitepartition}

%\begin{theorem} \label{bipartitepartition}
%Let $K_{n,n}$ be a two coloured balanced complete bipartite graph with classes $X$ and $Y$.  Either the colouring on $K_{n,n}$ is split or $K_{n,n}$ can be partitioned into two monochromatic paths of even order, and of different colours.
%\end{theorem}
\begin{proof} [Proof of Theorem \ref{bipartitepartition}]
Suppose that the colouring of $K_{n,n}$ is split.  Note that two monochromatic paths with different colours can intersect at most three of the sets $X_1$, $X_2$, $Y_1$ and~$Y_2$.  This together with the assumption that $X_1$, $X_2$, $Y_1$ and $Y_2$ are all nonempty implies that $K_{n,n}$ cannot be partitioned into two monochromatic paths with different colours.  

It remains to prove that every 2-coloured $K_{n,n}$ which is not split can be partitioned into two monochromatic paths.

The proof is by induction on $n$.  The case $n=1$ is trivial.  For the remainder of the proof assume that the result holds for $K_{m,m}$ for all $m<n$.

Assume that the colouring on $K_{n,n}$ is not split.  Lemma~\ref{split} gives us two cases to consider.  

Suppose that $K_{n,n}$ satisfies case (i) of Lemma~\ref{split}. Without loss of generality we can assume that $K_{n,n}$ is connected in {red}.

Apply Lemma~\ref{path-bipartiteconnected} to the {red} colour class of $K_{n,n}$.  If case (i) of Lemma~\ref{path-bipartiteconnected} occurs, then the theorem follows since we may choose $P$ to be our {red} path and the single vertex to be our {blue} path.  

So we can assume that we are in case (ii) of Lemma~\ref{path-bipartiteconnected}.  This gives us a partition of $K_{n,n}$ into a {red} path $P$, and three nonempty sets $A$, $B_1$, and $B_2$, such that $|A|=|B_1|+|B_2|$ and all the edges between $A$, $B_1$, and $B_2$ are {blue}.   Let $H = (A \cap X) \cup (B_1 \cap Y) \cup (B_2 \cap Y)$ and $K= (A \cap Y) \cup (B_1 \cap X) \cup (B_2 \cap X)$.  Note that $K_{n,n}[H]$ and $K_{n,n}[K]$ are both {blue} complete bipartite subgraphs of $K_{n,n}$, since all the edges between $A$ and $B_1\cup B_2$ are {blue}.  Notice that $|A|=|B_1|+|B_2|$  and $|X|=|Y|$ together imply that $P$ contains an even number of vertices.  This, together with the fact that the vertices of $P$ must alternate between $X$ and $Y$, implies that $|X\setminus P|=|Y\setminus P|$.  However $X\setminus P= X \cap (A\cup B_1\cup B_2)$ and $Y\setminus P= Y \cap (A\cup B_1\cup B_2)$, so we have that
\begin{equation} \label{path-bipartiteconnectedequation}
|X \cap A|+ |X\cap B_1| +|X \cap B_2|=|Y \cap A|+ |Y\cap B_1| +|Y \cap B_2|. 
\end{equation}
  Equation (\ref{path-bipartiteconnectedequation}), together with $|X \cap A|+|Y \cap A| = |Y\cap B_1| +|Y \cap B_2|+|X\cap B_1| +|X \cap B_2|$  implies that the following both hold:
\begin{align}
|A \cap X|&= |B_1 \cap Y|+|B_2 \cap Y|, \label{eq:equalityBY}\\
|A \cap Y|&= |B_1 \cap X|+|B_2 \cap X|. \label{eq:equalityBX}
\end{align}
Thus $K_{n,n}[H]$ and $K_{n,n}[K]$ are balanced {blue} complete bipartite subgraphs of $K_{n,n}$ and so can each be covered by a {blue} path.  If $H =\emptyset$ or $K=\emptyset$ holds, the theorem follows, since $V(K_{n,n})=V(P)\cup H \cup K$.   

So, we can assume that $H\neq \emptyset$ and $K\neq \emptyset$.  Equation (\ref{eq:equalityBY}), together with $H\neq\emptyset$, implies that $(B_1\cup B_2)\cap H \neq \emptyset$.  Similarly (\ref{eq:equalityBX}) together with $K\neq\emptyset$, implies that $(B_1\cup B_2)\cap K \neq \emptyset$.  We also know that $B_1$ and $B_2$ are nonempty and contained in $H\cup K$.  Combining all of these implies that at least one of the following holds.
\begin{enumerate}[(a)]
 \item $B_1 \cap H \neq \emptyset$ and $B_2 \cap K \neq \emptyset$.
 \item $B_1 \cap K \neq \emptyset$ and $B_2 \cap H \neq \emptyset$. 
\end{enumerate}
Suppose that (a) holds.  Choose $x \in B_1 \cap H$ and a {blue} path $Q$ covering $H$ and ending with $x$.  Choose $y \in B_2 \cap K$ and a {blue} path $R$ covering $K$ and starting with $y$.  Notice that $x\in X$ and $y\in Y$, so there is an edge  $xy$.  The edge $xy$  must be {blue} since it lies between $B_1$ and $B_2$.  This means that $Q+R$ is a {blue} path covering $A \cup B_1 \cup B_2=G\setminus P$, implying the theorem.  The case when (b) holds can be treated identically, exchanging the roles of $H$ and $K$.

\

Suppose that $K_{n,n}$ satisfies case (ii) of Lemma~\ref{split}.   Without loss of generality, this gives us a vertex $u \in X$ such that the edge $uy$ is {red} for every $y \in Y$. Let $v$ be any vertex in $Y$.

Suppose that the colouring of $K_{n,n}\setminus \{u,v\}$ is split with classes $X_1$, $X_2$, $Y_1$, and $Y_2$.  In this case $B(X_1,Y_2)$, $B(X_2, Y_1)$, and $\{v\}$ are all connected in {red}, and $u$ is connected to each of these by {red} edges.  This means that $K_{n,n}$ is connected in {red} and we are back to the previous case.  

So, suppose that the colouring of $K_{n,n}\setminus \{u,v\}$ is not split.  We claim that there is a partition of $K_{n,n}\setminus \{u,v\}$ into two a {red} path $P$ and a {blue} path $Q$ such that either $P$ is empty or $P$ ends in $Y$.  To see this, apply the inductive hypothesis to $K_{n,n}\setminus \{u,v\}$ to obtain a partition of this graph into a {red} path $P'$ and a {blue} path $Q'$.  If $P'$ is empty or $P'$ has an endpoint in $Y$, then we can let $P=P'$ and $Q=Q'$.   Otherwise, the endpoints $P'$ are in $X$, and so the endpoints of $Q'$ in $Y$.  Let $x$ be the end of $P'$ and $y$ the end of $Q'$.  If $xy$ is {red}, let $P=P'+\{y\}$ and $Q=Q' - \{y\}$. If $xy$ is {blue}, let $P=P' - \{x\}$ and $Q=Q'+ \{x\}$.  In either case, $P$ and $Q$ give a partition of $K_{n,n}\setminus \{u,v\}$ into two paths such that either $P$ is empty or $P$ has an endpoint in $Y$. 

Suppose that $P$ is empty.  In this case we have a partition of $K_{n,n}$ into a {red} path $\{u,v\}$ and a {blue} path $Q$.

Suppose that $P$ ends in a vertex, $w$, in $Y$.  The edges $uv$ and $uw$ are both {red}, so $P+\{u\}+\{v\}$ is a {red} path giving the required partition of $K_{n,n}$ into a {red} path $P+\{u\}+\{v\}$ and a {blue} path $Q$.
\end{proof}

As remarked in the introduction, there are split colourings of $K_{n,n}$ which cannot be partitioned into two monochromatic paths.  The following lemma shows that three monochromatic paths always suffice.

\begin{lemma} \label{splitthreepaths}
 Suppose that the edges of  $K_{n,n}$ are coloured with two colours.  Suppose that the colouring is split with classes $X_1$, $X_2$, $Y_1$, and $Y_2$.  For any two vertices $y_1 \in Y_1$ and $y_2 \in Y_2$, there is a vertex-partition of $K_{n,n}$ into a {red} path starting at $y_1$, a {red} path starting at $y_2$, and a {blue} path.
\end{lemma}
\begin{proof}
Without loss of generality, suppose that $X_1\leq X_2$ and $Y_1\leq Y_2$.  This, together with $X_1+X_2=Y_1+Y_2$ implies that  $X_1\leq Y_2$ and $Y_1\leq X_2$ both hold.  

$B(X_1, Y_2)$ is a {red} complete bipartite graph, so we can cover $X_1$ and $|X_1|$ vertices in $Y_2$ with a {red} path starting from $y_1$.  Similarly we can cover  $Y_2$ and $|Y_2|$ vertices in $X_1$ with a {red} path starting from $y_2$.   The only uncovered vertices are in $Y_2$ and $X_1$.  All the edges between these are {blue}, so we can cover the remaining vertices with a {blue} path.
\end{proof}

The following lemma gives an alternative characterization of $4$-partite colourings of $K_n$.

\begin{lemma} \label{4partite}
Suppose that the edges of $K_n$ are coloured with three colours.  The colouring is 4-partite if and only if there is a {red} connected component $C_1$ and a {blue} connected component $C_2$ such that all of the sets $C_1\cap C_2$, $(V(K_n)\setminus C_1)\cap C_2$, $C_1\cap (V(K_n)\setminus C_2)$, and $(V(K_n)\setminus C_1)\cap (V(K_n)\setminus C_2)$ are nonempty.
\end{lemma}

\begin{proof}
Suppose that we have a {red} component $C_1$ and a {blue} component $C_2$ as in the statement of the lemma.  Let  $A_1=C_1\cap (V(K_n)\setminus C_2)$, $A_2=(V(K_n)\setminus C_1)\cap C_2$, $A_3=(V(K_n)\setminus C_1)\cap (V(K_n)\setminus C_2)$, and $A_4=C_1\cap C_2$.  This ensures that the sets $A_1$, $A_2$, $A_3$, and $A_4$ form the classes of a 4-partite colouring of $K_n$.

For the converse, suppose that $A_1$, $A_2$, $A_3$, and $A_4$  form the classes of a 4-partite colouring.  Choose $C_1=A_1\cup A_4$ and $C_2=A_2 \cup A_4$ to obtain components as in the statement of the lemma.  
\end{proof}

We are now ready to prove Theorem~\ref{not4partitecase}.
%\begin{lemma} \label{not4partitecase}
%Suppose that the edges of $K_n$ are coloured with three colours such that the colouring is not 4-partite.  Then there is a partition of $K_n$ into three monochromatic paths of different colours.
%\end{lemma}
\begin{proof}[Proof of Theorem \ref{not4partitecase}]
 The two main cases that we will consider are when $K_n$ is connected in some colour, and when $K_n$ is disconnected in all three colours.

Suppose that $K_n$ is connected in {red}.  Apply Lemma~\ref{path-bipartiteconnected} to the {red} colour class of $K_n$.  If case (i) of Lemma~\ref{path-bipartiteconnected} occurs, then the theorem follows since we can take $P$ as our {red} path, the single vertex as our {blue} path and the empty set as our {green} path.  so, suppose that case (ii) of Lemma~\ref{path-bipartiteconnected} occurs, giving us a partition of $K_n$ into a {red} path $P$ and three sets $A$, $B_1$, and $B_2$ such that $|A|=|B_1|+|B_2|$ and all the edges between $A$, $B_1$, and $B_2$ are {blue} or {green}. 

If the colouring on $B(A, B_1\cup B_2)$ is not split, we can apply Theorem~\ref{bipartitepartition} to partition $B(A, B_1\cup B_2)$  into a {blue} path and a {green} path proving the theorem.

So, assume $B(A, B_1\cup B_2)$ is split with classes $X_1$, $X_2$, $Y_1$, and $Y_2$, such that $A=X_1 \cup X_2$ and $B_1\cup B_2= Y_1\cup Y_2$.
Then, the fact that $B_1$, $B_2$, $Y_1$, and $Y_2$ are nonempty implies that one of the following holds.
\begin{enumerate} [(i)]
 \item $B_1\cap Y_1 \neq \emptyset$ and $B_2\cap Y_2 \neq \emptyset$.
 \item $B_1\cap Y_2 \neq \emptyset$ and $B_2\cap Y_1 \neq \emptyset$.
\end{enumerate}

Assume that (i) holds.  Choose $y_1\in B_1\cap Y_1$ and $y_2 \in B_2\cap Y_2$.  The edge $y_1 y_2$ must be {blue} or {green} since it lies between $B_1$ and $B_2$.  Assume that $y_1 y_2$ is {blue}.  Apply Lemma~\ref{splitthreepaths} to partition $B(A, B_1\cup B_2)$ into a {blue} path $Q$ ending with $y_1$, a {blue} path $R$ starting from $y_2$ and a {green} path $S$.  By joining $Q$ and $R$, we obtain a partition of $G$ into three monochromatic paths $P$, $Q + R$, and $S$, all of different colours.  The cases when (ii) holds or when the edge $y_1 y_2$ is {green} are dealt with similarly.

\

The same argument can be used if $K_n$ is connected in {blue} or {green}.  So, for the remainder of the proof, we assume that all the colour classes are disconnected.  Let $C$ be the largest connected component in any colour class.  Without loss of generality we may suppose that $C$ is a {red} connected component.  Let $D$ be a {blue} connected component.  Let $C^c=V(K_n)\setminus C$ and $D^c=V(K_n)\setminus D$.  One of  the sets $C\cap D$, $C^c\cap D$, $C\cap D^c$, or $C^c\cap D^c$ must be empty.  Indeed if all these sets were nonempty, then Lemma~\ref{4partite} would imply that the colouring is 4-partite, contradicting the assumption of the theorem.

We claim that $D\subseteq C$ or  $D^c\subseteq C$ holds.  To see this consider four cases depending on which of $C\cap D$, $C^c\cap D$, $C\cap D^c$ or $C^c\cap D^c$ is empty.
\begin{itemize}
 \item $C\cap D = \emptyset$ implies that all the edges between $C$ and $D$ are {green}.  This contradicts $C$ being the largest component in any colour.
 \item $C^c\cap D = \emptyset$ implies that $D\subseteq C$.
 \item $C\cap D^c = \emptyset$ implies that $C\subseteq D$.  Since $C$ is the largest component of any colour, this means that $C=D$.  
 \item $C^c\cap D^c = \emptyset$ implies that $D^c \subseteq C$.
\end{itemize}

If $D\subseteq C$ holds, then choose $v\in D$.  If $D^c\subseteq C$ holds, then choose $v \in D^c$.  In either case all the edges between $v$ and $C^c$ must be {green}.

Apply Lemma~\ref{path-bipartite} to the {red} colour class of $K_n$ in order to obtain a partition of $K_n$ into a {red} path $P$ and two sets $A$ and $B$ such that $|A|=|B|$ and all the edges between $A$ and $B$ are colours $2$ or $3$.  In addition, $P$ is either empty or starts at $v$.  If either of the graphs $K_n[A]$ or $K_n[B]$ is disconnected in {red}, then we can proceed just as we did after we applied Lemma~\ref{path-bipartiteconnected} in the previous part of the theorem.  So assume that both $K_n[A]$ and $K_n[B]$ are connected in {red}.  We claim that one of the sets $A$ or $B$ must be contained in $C^c$.  Indeed otherwise $C$ would intersect each of $P$, $A$, and $B$.  Since $P$, $K_n[A]$, and $K_n[B]$ are connected in {red}, this would imply that $C= P\cup A \cup B =K_n$ contradicting $K_n$ being disconnected in {red}.  Without loss of generality we may assume that $B \subseteq C^c$.  Therefore all the edges between $v$ and $B$ are  {green}.  

As before, if  the colouring on $B(A, B)$ is not split, we can apply Theorem~\ref{bipartitepartition} to partition $B(A, B)$  into a {blue} path and a {green} path.  Therefore assume that the colouring on $B(A,B)$ is split.

If the path $P$ is empty, then we must have $v \in A$.  Lemma \ref{split} leads to a contradiction, since we know that all the edges between $v$ and $B$ are {green}, and $B(A, B)$ is split.  

Therefore the path $P$ is nonempty.  We know that $B(A, B)$ is split with classes $X_1$, $X_2$, $Y_1$, and $Y_2$, such that $A=X_1 \cup X_2$ and $B= Y_1\cup Y_2$.
Choose $y_1\in Y_1$ and $y_2 \in Y_2$ arbitrarily.  Apply Lemma~\ref{splitthreepaths} to $B(A,B)$ to partition $B(A,B)$ into a {green} path $Q$ ending with $y_1$, a {green} path $R$ starting from $y_2$, and a {blue} path $S$.  Notice that the edges $y_1v$ and $vy_2$ are both {green}, so $P -\{v\}$, $S$, and $Q + \{v\} + R$ give a partition of $K_n$ into three monochromatic paths, all of different colours.

\end{proof}

\subsection{Proof of Theorem~\ref{4partitecase}}
In this section, we prove Theorem~\ref{4partitecase}.  The proof has a lot in common with the proof of a similar theorem in~\cite{Szemeredi1}. 

\begin{proof}
Let $A_1$, $A_2$, $A_3$, and $A_4$ be the classes of the 4-partition of $K_n$, with colours between the classes as in Definition~\ref{4partitedefinition}.    Our proof will be by induction on $n$.  The initial case of the induction will be $n=4$, since for smaller $n$ there are no 4-partite colourings of $K_n$.  For $n=4$, the result is trivial.  Suppose that the result holds for $K_m$ for all $m< n$. 

For $i=1$, $2$, $3$, and $4$ we assign three integers $r_i$, $b_i$, and $g_i$ to $A_i$ corresponding to the three colours as follows:
\begin{enumerate}[(i)]
\item Suppose that $A_i$ can be partitioned into three nonempty monochromatic paths $R_i$, $B_i$, and $G_i$ of colours {red}, {blue}, and {green} respectively.  In this case, let $r_i=|R_i|$, $b_i=|B_i|$, and $g_i=|G_i|$.
\item Suppose that $A_i$ can be partitioned into three nonempty monochromatic paths $P_1$, $P_2$, and $Q$ such that $P_1$ and $P_2$ are coloured the same colour and $Q$ is coloured a different colour.  
If $P_1$ and $P_2$ are {red}, then we let $r_i=|P_1|+|P_2|-1$.  If $Q$ is {red}, then we let $r_i=|Q|$.  If none of $P_1$, $P_2$, or $Q$ are {red}, then we let $r_i=1$.  We do the same for ``{blue}" and ``{green}" to assign values to $b_i$ and $g_i$ respectively.  As a result we have assigned the values $|P_1|+|P_2|-1$, $|Q|$, and $1$ to some permutation of the three numbers $r_i$, $b_i$, and $g_i$. 
\item Suppose that $|A_i|\leq 2$.  In this case, let $r_i=b_i=g_i=1$.
\end{enumerate}

For each $i \in \{1,2,3,4\}$, $A_i$ will always be in at least one of the above three cases.  To see this, depending on whether the colouring on $A_i$ is $4$-partite or not, apply either Theorem \ref{not4partitecase} or the inductive hypothesis of Theorem \ref{4partitecase} to $A_i$, in order to partition $A_i$ into three monochromatic paths $P_1$, $P_2$, and $P_3$ at most two of which are the same colour.  If $|A_i|\geq 3$ then we can assume that  $P_1$, $P_2$, and $P_3$ are nonempty.  Indeed if $P_1$, $P_2$, or $P_3$ are empty, then we can remove endpoints from the longest of the three paths and add them to the empty paths to obtain a partition into three nonempty paths, at most two of which are the same colour.  Therefore, if $|A_i|\geq 3$, then either Case (i) or (ii) above will hold, whereas if $A_i\leq 2$, then Case (iii) will hold.

For each $i \in \{1,2,3,4\}$, note that $r_i$, $b_i$, and $g_i$ are positive and satisfy $r_i + b_i + g_i \geq |A_i|$.
We will need the following definition.

%\begin{align}
%r_i, ` b_i, ` g_i &\geq 1 ,\label{eq:aipositive}\\ 
%r_i + b_i + g_i &\geq |A_i|. \label{eq:aisum}
%\end{align}

\begin{definition}
A {red} linear forest $F$ is $A_i$-\textbf{filling} if $F$ is contained in $A_i$, and either $F$ consists of one path of order $r_i$, or $F$ consists of two paths $F_1$ and $F_2$ such that $|F_1|+|F_2|=r_i+1$.
\end{definition}

{Blue} or {green} $A_i$-filling linear forests are  defined similarly, exchanging the role of $r_i$ for $b_i$ or $g_i$ respectively.  We will need the following two claims.

\begin{claim}\label{twoforests}
Suppose that $i\in \{1,2,3,4\}$, and $|A_i|\geq 2$.  There exist two disjoint $A_i$-filling linear forests with different colours for any choice of two different colours.
\end{claim}
\begin{proof}
Claim \ref{twoforests} holds trivially from the definition of $r_i$, $b_i$, and $g_i$.
\end{proof}

\begin{claim}\label{onepath}
Suppose that $i,j\in \{1,2,3,4\}$ such that $i\neq j$ and $B(A_i, A_j)$ is {red}. Let $m$ be an integer such that the following hold.
\begin{align}
0\leq m &\leq r_i\label{eq:onepath1},\\
|A_i|-m &\leq  |A_j|. \label{eq:onepath2}
\end{align}
There exists a {red} path $P$ from $A_i$ to $A_i$, of order $2|A_i|-m$, covering all of $A_i$ and any set of $A_i-m$ vertices in $A_j$.
\end{claim}

\begin{proof} 

Note that we can always find an $A_i$-filling linear forest, $F$.  If $|A_i|=1$ this is trivial, and if $|A_i|\geq 2$, then this follows from Claim \ref{twoforests}.

Suppose that $F$ consists of one path of order $r_i$. By (\ref{eq:onepath1}), we can shorten $F$ to obtain a new path $F'$ of order $m$.  By (\ref{eq:onepath2}), we can choose a {red} path, $P$, from $A_i$ to $A_j$ consisting of $A_i\setminus F'$ and any $|A_i|-m$ vertices in $A_j \setminus F'$.  The path $P+F'$ satisfies the requirements of the claim.

Suppose that $F$ consists of two paths $F_1$ and $F_2$ such that $|F_1|+|F_2|=r_i+1$.  By (\ref{eq:onepath1}), we can shorten $F_1$ and $F_2$ to obtain two paths $F_1'$ and $F_2'$ such that $|F_1|+|F_2|=m+1$.    By (\ref{eq:onepath2}), we can choose a {red} path, $P$, from $A_i$ to $A_j$ consisting of $A_i \setminus F'$ and any $|A_i|-m-1$ vertices in $A_j$. By (\ref{eq:onepath2}) there must be at least one vertex, $v$, in $A_j\setminus P$.  The path $P+F_1+\{v\}+F_2 $ satisfies the requirements of the claim.

\end{proof}

We can formulate versions of Claim \ref{onepath} for the colours {blue} or {green} as well, replacing $r_i$ with $b_i$ or $g_i$ respectively.

To prove Theorem \ref{4partitecase}, we will consider different combinations of values of $r_i$, $b_i$, and $g_i$ for $i = 1$, $2$, $3$, and $4$ to construct a partition of $K_n$ into three monochromatic paths in each case.  

  If a partition of $K_n$ into monochromatic paths contains edges in the graph $B(A_i, A_j)$ for $i \neq j$, we say that $B(A_i, A_j)$ is a \emph{target component} of the partition.  Note that a partition of $K_n$ into three monochromatic paths can have at most three target components.  This is because a monochromatic path can pass through edges in at most one of graphs $B(A_i, A_j)$.  

There are two kinds of partitions into monochromatic paths which we shall construct.  
\begin{itemize}
 \item We say that a partition of $K_n$ is \emph{star-like} if the target components are $B(A_i,A_j)$, $B(A_i,A_k)$, and $B(A_i, A_l)$, for $(i,j,k,l)$ some permutation of $(1,2,3,4)$.  In this case, all the paths in the partition will have different colours.
\item We say that a partition of $K_n$ is \emph{path-like} if the target components are $B(A_i,A_j)$, $B(A_j,A_k)$, and $B(A_k, A_l)$, for $(i,j,k,l)$ some permutation of $(1,2,3,4)$.  In this case, two of the paths in the partition will have the same colour.
\end{itemize}

For $i\in \{1,2,3,4\}$, it is possible to write down sufficient conditions on $|A_i|$, $r_i$, $b_i$, and $g_i$ for $K_n$ to have a partition into three monochromatic paths with given target components.

\begin{claim} \label{star}
Suppose that the following holds: 
\begin{equation} \label{eq:star}
|A_1| + |A_2| + |A_3| \leq |A_4| + r_1 + b_2 + g_3. 
\end{equation}
 Then, $K_n$ has a star-like partition with target components $B(A_4,A_1)$, $B(A_4,A_2)$, and $B(A_4,A_3)$ of colours {red}, {blue}, and {green} respectively.
\end{claim}

\begin{proof}
Using (\ref{eq:star}), we can find three disjoint subsets $S_1$, $S_2$, and $S_3$ of $A_4$ such that  $|S_1|=|A_1|-r_1$, $|S_2|=|A_2|-b_2$, and $|S_3|=|A_3|-g_3$ all hold.  By Claim \ref{onepath} there is a {red} path $P_1$ with vertex set  $A_1\cup S_1$, a {blue} path $P_2$ with vertex set $A_2\cup S_2$, and a {green} path $P_3$ with vertex set $A_3\cup S_3$.  The paths $P_1$, $P_2$, and $P_3$ are pairwise disjoint and have endpoints in $A_1$, $A_2$, and $A_3$ respectively.

Depending on whether $A_4\setminus (P_1 \cup P_2 \cup P_3)$ is $4$-partite or not, apply either Theorem~\ref{not4partitecase} or the inductive hypothesis to find a partition of $A_4\setminus (P_1 \cup P_2 \cup P_3)$ set into three monochromatic paths $Q_1$, $Q_2$, and $Q_3$ at most two of which are the same colour.

We will join the paths $P_1$, $P_2$, $P_3$, $Q_1$, $Q_2$, and $Q_3$ together to obtain three monochromatic paths partitioning all the vertices in $K_n$.  

Suppose that all the $Q_i$ are all of different colours, with $Q_1$ {red}, $Q_2$ {blue}, and $Q_3$ {green}.  In this case $P_1+Q_1$, $P_2+Q_2$, and $P_3+Q_3$ are three monochromatic paths forming a star-like partition of $K_n$.

Suppose that two of the $Q_i$ are the same colour.  Without loss of generality, we may assume that $Q_1$ and $Q_2$ are {red} and $Q_3$ is {blue}.  In this case $Q_1+P_1+Q_2$, $P_2$, and $P_3+Q_3$ are three monochromatic paths forming a star-like partition of $K_n$.
\end{proof}

\begin{claim} \label{path}
Suppose that the following all hold:
\begin{align}
|A_1|+|A_4| &\leq |A_2|+|A_3| + b_4 + g_4+ g_1 \label{eq:path1},\\
|A_3|+|A_2| &\leq |A_1|+|A_4| + b_2 + g_2+ g_3 \label{eq:path2}, \\
|A_1| &< |A_2| + g_1 \label{eq:path3},\\
|A_3| &< |A_4| + g_3 \label{eq:path4}.
\end{align}
Then $K_n$ has a path-like partition with target components $B(A_1,A_2)$, $B(A_2,A_4)$, and $B(A_4,A_3)$ of colours {green}, {blue}, and {green} respectively.
\end{claim}
\begin{proof}
 Suppose that we have
\begin{equation} \label{eq:path5}
|A_2|-|A_1|+g_1\geq |A_4|-|A_3|+g_3.
\end{equation}
 The inequality (\ref{eq:path4}), together with Claim~\ref{onepath} ensures that we can find a {green} path $P_1$ consisting of all of $A_3$ and  $|A_3|-g_3$ vertices in $A_4$.  

\

There are two subcases depending on whether the following holds or not:
\begin{equation} \label{eq:path6}
|A_2|-|A_1|\leq |A_4|-|A_3|+g_3.
\end{equation}

 Suppose that (\ref{eq:path6}) holds.  
Let $m= |A_1|-|A_2|+ |A_4|-|A_3|+g_3$.  Note that $|A_1|-m\leq |A_2|$ holds by (\ref{eq:path4}), that  $m$ is positive by (\ref{eq:path6}), and that $m$ is less than $g_1$ by (\ref{eq:path5}).
Therefore, we can apply Claim \ref{onepath} to find a {green} path $P_2$  consisting of $A_1$ and $|A_1|-m$ vertices in $A_2$.  
 There remain exactly $|A_4|-|A_3|+g_3$ vertices in each of $A_2$ and $A_4$ outside of the paths $P_1$ and $P_2$.  Cover these with a {blue} path $P_3$ giving the required partition.

\

Suppose that (\ref{eq:path6}) fails to hold. 
Note that (\ref{eq:path4}) and the negation of (\ref{eq:path6}) imply that $|A_2|> |A_1|$ which, together with the fact that $|A_1|>0$, implies that $|A_2|\geq 2$.
Therefore, we can apply Claim \ref{twoforests} to $A_2$ to obtain a {blue} $A_2$-filling linear forest, $B$, and a disjoint {green} $A_2$-filling linear forest, $G$.
We construct a {blue} path $P_B$ and a {green} path $P_G$ as follows:

Note that $A_4\setminus P_1$ is nonempty by (\ref{eq:path4}), so let $u$ be a vertex in $A_4\setminus P_1$.  If $B$ is the union of two paths $B_1$ and $B_2$ such that $|B_1|+|B_2|=b_2+1$, then let $P_B=B_1+\{v\}+B_2$.  Otherwise $B$ must be single path of order $b_1$, and we let $P_B=B$.

Similarly, let $v$ be a vertex in $A_1$.  If $G$ consists of two paths $G_1$ and $G_2$, we let $P_G=G_1+\{v\}+G_2$.  If $G$ is a single path, we let $P_G=G$.

Note that the above construction and (\ref{eq:path2}) imply that the following is true.

\begin{equation}\label{eq:path7}
|A_2\setminus (P_B\cup P_G)|\leq |A_1\setminus (P_B\cup P_G)| + |(A_4 \setminus P_1) \setminus (P_B\cup P_G)|.
\end{equation}

The negation of (\ref{eq:path6}) is equivalent to the following 

\begin{equation}\label{eq:path8}
|A_2| \geq |A_1| + |(A_4 \setminus P_1)|.
\end{equation}

Let $P_B'$ and $P_G'$ be subpaths of $P_B$ and $P_G$ respectively, such that the sum $|P_B'|+|P_G'|$ is as small as possible and we have
\begin{equation}\label{eq:path9}
|A_2\setminus (P_B'\cup P_G')|\leq |A_1\setminus (P_B'\cup P_G')| + |(A_4 \setminus P_1') \setminus (P_B'\cup P_G')|.
\end{equation}
The paths $P_B'$ and $P_G'$ are well defined by (\ref{eq:path7}). We claim that we actually have equality in (\ref{eq:path9}).  Indeed, since $A_1$, $A_2$, and $A_4 \setminus P_1'$ are all disjoint, removing a single vertex from $P_B'$ or $P_G'$ can change the inequality (\ref{eq:path9}) by at most one.  Therefore, if the inequality (\ref{eq:path9}) is strict, we know that $P_B'$ and $P_G'$ are not both empty by (\ref{eq:path8}), so we can always remove a single vertex from $P_B'$ or $P_G'$ to obtain shorter paths satisfying (\ref{eq:path9}), contradicting the minimality of $|P_B'|+|P_G'|$.

Equality in (\ref{eq:path9}) implies that $|A_2\setminus(P_1\cup P_B'\cup P_G')|=|(A_1\cup A_4)\setminus(P_1\cup P_B'\cup P_G')|$, so we can choose a {green} path $Q_G$ from $A_1$ to $A_2$ and a disjoint {blue} path $Q_B$ from $(A_4 \setminus P_1)$ to $A_2$ such that $Q_B\cup Q_G=(A_1\cup A_2 \cup A_4)\setminus (P_1\cup P_B'\cup P_G')$.  The paths $P_1$, $P_B'+Q_B$, and $P_G'+Q_G$ give us the required partition of $K_n$.

\

If the negation of (\ref{eq:path5}) holds, we can use the same method, exchanging the roles of $A_1$ and $A_3$, and of $A_2$ and $A_4$.

\end{proof}

Obviously, there was nothing special about our choice of target components in Claims~\ref{star} and~\ref{path}.  We can write similar sufficient conditions for there to be a star-like partition or a path-like partition for any choice of target components.  
To prove Theorem~\ref{4partitecase}, we shall show that either the inequality in Claim~\ref{star} holds or all four inequalities from Claim~\ref{path} hold for some choice of target components.  

Without loss of generality we can assume that the following holds:
\begin{equation}\label{eq:order2}
|A_1|\leq|A_2|\leq|A_3|\leq|A_4|.
\end{equation}
Consider the following instances of Claims~\ref{star} and~\ref{path}.

There is a star-like partition of $K_n$ with target components $B(A_4,A_1)$, $B(A_4,A_2)$, and $B(A_4, A_3)$ of colours {red}, {blue}, and {green} respectively if the following holds:
\begin{equation*} \label{eq:4star}
 |A_1| + |A_2| + |A_3| \leq |A_4| + r_1 + b_2 + g_3. \tag{A1}
\end{equation*}

There is a path-like partition of $K_n$ with target components $B(A_1,A_2)$, $B(A_2,A_4)$, and $B(A_4, A_3)$ of colours {green}, {blue}, and {green} respectively if the following holds:
\begin{align*}
|A_1|+|A_4| &\leq |A_2|+|A_3| + b_4 + g_4+ g_1, \label{eq:1243in1} \tag{B1}\\
|A_3|+|A_2| &\leq |A_1|+|A_4| + b_2 + g_2+ g_3, \label{eq:1243in2} \tag{B2}\\
|A_1| &< |A_2| + g_1, \\%\label{eq:1243in3} \\
|A_3| &< |A_4| + g_3.   %\label{eq:1243in4} 
\end{align*}

There is a path-like partition of $K_n$ with target components $B(A_1,A_4)$, $B(A_4,A_3)$, and $B(A_3, A_2)$ of colours {red}, {green}, and {red} respectively if the following holds:
\begin{align*} 
|A_1|+|A_3| &\leq |A_2|+|A_4| + g_3 + r_3+ r_1, \\%\label{eq:1432in1}\\
|A_2|+|A_4| &\leq |A_1|+|A_3| + g_4 + r_4+ r_2, \label{eq:1432in2} \tag{C2}\\
|A_1| &< |A_4| + r_1, \\%\label{eq:1432in3}\\
|A_2| &< |A_3| + r_2. \\%\label{eq:1432in4}
\end{align*}

There is a path-like partition of $K_n$ with target components $B(A_1,A_3)$, $B(A_3,A_2)$, and $B(A_2, A_4)$ of colours {blue}, {red}, and {blue} respectively if the following holds:
\begin{align*}
|A_1|+|A_2| &\leq |A_3|+|A_4| + r_2 + b_2+ b_1, \\%\label{eq:1324in1}\\
|A_3|+|A_4| &\leq |A_1|+|A_2| + r_3 + b_3+ b_4, \label{eq:1324in2} \tag{D2}\\
|A_1| &< |A_3| + b_1, \\%\label{eq:1324in3} \\
|A_4| &< |A_2| + b_4. \label{eq:1324in4} \tag{D4}
\end{align*}

There is a path-like partition of $K_n$ with target components $B(A_1,A_4)$, $B(A_4,A_2)$, and $B(A_2, A_3)$ of colours {red}, {blue}, and {red} respectively if the following holds:
\begin{align*}
|A_1|+|A_2| &\leq |A_3|+|A_4| + b_2 + r_2+ r_1, \\%\label{eq:1423in1}\\
|A_3|+|A_4| &\leq |A_1|+|A_2| + b_4 + r_4+ r_3, \label{eq:1423in2} \tag{E2}\\
|A_1| &< |A_4| + r_1, \\%\label{eq:1423in3}\\
|A_3| &< |A_2| + r_3. \label{eq:1423in4} \tag{E4}
\end{align*}

There is a path-like partition of $K_n$ with target components $B(A_2,A_4)$, $B(A_4,A_1)$, and $B(A_1, A_3)$ of colours {blue}, {red}, and {blue} respectively if the following holds:
\begin{align*}
|A_1|+|A_2| &\leq |A_3|+|A_4| + r_1 + b_1+ b_2, \\%\label{eq:2413in1}\\
|A_3|+|A_4| &\leq |A_1|+|A_2| + r_4 + b_4+ b_3, \label{eq:2413in2} \tag{F2}\\
|A_2| &< |A_4| + b_2, \\%\label{eq:2413in3}\\
|A_3| &< |A_1| + b_3. \label{eq:2413in4} \tag{F4}
\end{align*}

%Note that (\ref{eq:1243in3}), (\ref{eq:1243in4}), (\ref{eq:1432in1}), (\ref{eq:1432in3}), (\ref{eq:1432in4}), (\ref{eq:1324in1}), (\ref{eq:1324in3}), (\ref{eq:1423in1}), (\ref{eq:1423in3}), (\ref{eq:2413in1}), (\ref{eq:2413in3}) all hold as a consequence of  $|A_1|\leq|A_2|\leq|A_3|\leq|A_4|$.

Note that all the unlabeled inequalities hold as a consequence of (\ref{eq:order2}) and the positivity of $r_i$, $b_i$, and $g_i$.  Thus, to prove the theorem  it is sufficient to show that all the labeled inequalities corresponding to some particular letter A, B, C, D, E, or F hold.  
We split into two cases depending on whether (\ref{eq:1243in1}) holds or not.

\textbf{Case 1:}  Suppose that (\ref{eq:1243in1}) holds.

Note that the following cannot all be true at the same time:

\begin{align}
|A_3| + r_2 &> |A_4| + g_3, \label{eq:case1243} \\
|A_2| + b_4 &> |A_3| + r_2, \label{eq:case1432} \\
|A_4| + g_3 &> |A_2| + b_4. \label{eq:case1324} 
\end{align}
Indeed adding these three inequalities together gives $0>0$.  Thus the negation of (\ref{eq:case1243}), (\ref{eq:case1432}), or (\ref{eq:case1324}) must hold.

The negation of (\ref{eq:case1243}) implies (\ref{eq:1243in2}) which, together with our assumption that (\ref{eq:1243in1}) holds, implies that all the inequalities corresponding to the letter ``B'' hold.

The negation of (\ref{eq:case1432}) implies (\ref{eq:1432in2}) which implies that all the inequalities corresponding to the letter ``C'' hold.

The negation of (\ref{eq:case1324}), together with $|A_3|\leq r_3+b_3+g_3$ implies that (\ref{eq:1324in2}) holds.  The negation of (\ref{eq:case1324}), together with $g_3>0$ implies that (\ref{eq:1324in4}) holds. Therefore, all the inequalities corresponding to the letter ``D'' hold.  

\

\textbf{Case 2:}  Suppose that (\ref{eq:1243in1}) does not hold.  If (\ref{eq:1432in2}) holds, then all the inequalities labeled ``C'' hold, so we assume that the negation of (\ref{eq:1432in2}) holds. We consider three subcases depending on which of (\ref{eq:1423in4}) and (\ref{eq:2413in4}) hold.

\emph{Subcase 1:}  Suppose that (\ref{eq:1423in4}) holds.  If (\ref{eq:1423in2}) holds, then all the inequalities labeled ``E'' hold, so we assume that the negation of (\ref{eq:1423in2}) holds.  Adding the negations of (\ref{eq:1243in1}), (\ref{eq:1432in2}), and (\ref{eq:1423in2}) together, and using $|A_4| \leq r_4+b_4+g_4$ gives the following:
\begin{equation*} \label{eq:caseE}
%3|A_4|  > |A_1|+|A_2|+|A_3| + 2(r_4+b_4+g_4)+ g_1+r_2+r_3
|A_4|  > |A_1|+|A_2|+|A_3| + g_1+r_2+r_3.
\end{equation*}
This is stronger than (\ref{eq:4star}) which implies that all the inequalities corresponding to the letter ``A'' hold.

\

\emph{Subcase 2:}  Suppose that (\ref{eq:2413in4}) holds.  If (\ref{eq:2413in2}) holds, then all the inequalities labeled ``F'' hold, so we assume that the negation of (\ref{eq:2413in2}) holds.  Adding the negations of (\ref{eq:1243in1}), (\ref{eq:1432in2}), and (\ref{eq:2413in2}) together, and using $|A_4|\leq r_4+b_4+g_4$ gives the following:
\begin{equation*} \label{eq:caseF}
%3|A_4|  > |A_1|+|A_2|+|A_3| + 2(r_4+b_4+g_4)+ g_1+r_2+r_3
|A_4|  > |A_1|+|A_2|+|A_3| + g_1+r_2+b_3.
\end{equation*}
This is stronger than (\ref{eq:4star}) which implies that all the inequalities corresponding to the letter ``A'' hold.

\

\emph{Subcase 3:}  Suppose that neither (\ref{eq:1423in4}) or (\ref{eq:2413in4}) hold.  Adding the negations of (\ref{eq:1243in1}), (\ref{eq:1432in2}), (\ref{eq:1423in4}), and (\ref{eq:2413in4}) together, and using $|A_4|\leq r_4+b_4+g_4$ and $|A_3|\leq r_3+b_3+g_3$  gives the following:
\begin{equation*} \label{eq:caseNEF}
|A_4| + g_3 > |A_1|+|A_2|+|A_3| + g_1+r_2+g_4.
\end{equation*}
This is stronger than (\ref{eq:4star}) which implies that all the inequalities corresponding to the letter ``A'' hold.
\end{proof}

\section{Discussion}
Much of the research on partitioning coloured graphs has focused around Conjecture~\ref{Erdos}.  Given the disproof of this conjecture, we will spend the remainder of this paper discussing possible directions for further work.

Although we only constructed counterexamples to Conjecture~\ref{Erdos} for particular $n$ in Section 2 of this paper, it is easy to generalize our construction to work for all $n\geq N_r$, where $N_r$ is a number depending on $r$.  To see this, one only needs to replace the assumption of ``$m$ is an integer"  with ``$m$ is a real number" in Section 2, and replace expressions where $m$ appears with suitably chosen integral parts.  Doing this and choosing $m$ appropriately will produce $r$-colourings of $K_n$ which cannot be partitioned into $r$ monochromatic cycles for all sufficiently large $n$.

A weakening of Conjecture~\ref{Erdos} is the following approximate version.
\begin{conjecture}\label{approximate}
For each $r$ there is a constant $c_r$, such that in every $r$-edge coloured complete graph $K_n$, there are $r$ vertex-disjoint monochromatic cycles covering $n - c_r$ vertices in $K_n$.  
\end{conjecture}
This conjecture is open for $r\geq 3$.  For $r=3$, Theorem~\ref{threecycles} shows that a version of Conjecture~\ref{approximate} is true with $c_r$ replaced with a function $o_r(n)$ satisfying $\frac{o_r(n)}{n}\to 0$ as $n\to \infty$.

Another way to weaken Conjecture~\ref{Erdos} is to remove the constraint that the cycles covering $K_n$ are disjoint.
\begin{conjecture}\label{covering}
Suppose that the edges of $K_n$ are coloured with $r$ colours.  There are $r$ (not necessarily disjoint) monochromatic cycles covering all the vertices in $K_n$.  
\end{conjecture}
A weaker version of this conjecture where ``cycles" is replaced with ``paths" has appeared in~\cite{Gyarfas}.  Our method of finding counterexamples to Conjecture~\ref{Erdos} relied on first finding graphs which cannot be partitionined into $r$ monochromatic paths of different colours.  For $r=3$, using Theorem~\ref{not4partitecase}, it is easy to show that every 3-coloured complete graph can be covered by three (not necessarily disjoint) paths of different colours.  Therefore for $r=3$, it is unlikely that something similar to our constructions in Section~2 can produce counterexamples to Conjecture~\ref{covering}.  It is even possible that, for all $r$, one can ask for the cycles in Conjecture~\ref{covering} to have different colours.

As mentioned in the introduction, it is interesting to consider partitions of an edge coloured graph $G$ other than the complete graph.  Theorems~\ref{GyarfasLehel} and~\ref{bipartitepartition} are results in this direction when when $G$ is a balanced complete bipartite graph.  We make the following conjecture which would generalise Corollary~\ref{bipartitepartitionsimple}.
\begin{conjecture}\label{bipartiteconjecture}
Suppose that the edges of $K_{n,n}$ are coloured with $r$ colours.  There is a vertex-partition of $K_{n,n}$ into $2r-1$  monochromatic paths.  
\end{conjecture}
This conjecture would be optimal, since for all $r$, there exist $r$-coloured balanced complete bipartite graphs which cannot be partitioned into $2r-2$ monochromatic paths.  
We sketch one such construction here.  Let $X$ and $Y$ be the classes of the bipartition of a balanced complete bipartite graph.  We partition $X$ into $X_1,\dots, X_r$ and $Y$ into $Y_1,\dots, Y_r$ where $|X_i|=10^i+i$ and $|Y_i|=10^i+r-i$.  The edges between $X_i$ and $Y_j$ are coloured with colour $i+j\pmod{r}$.  It is possible to show that this graph cannot be partitioned into $2r-2$ monochromatic paths.
In~\cite{Haxell}, Haxell showed that every $r$-edge coloured balanced complete bipartite graph can be partitioned into $O((r\log{r})^2)$ monochromatic cycles, which is the best known upper bound to how many paths are needed in Conjecture~\ref{bipartiteconjecture}.

%For $r=2$,  split colourings achieve this, as remarked in the introduction.  For $r\geq 3$, we briefly sketch a construction of $r$-coloured balanced complete bipartite graphs which cannot be partitioned into $2r-2$ monochromatic paths.  Let $X$ and $Y$ be the classes of the bipartition of $K_{n,n}$.  Partition $X$ into two sets $X_1$ and $X_2$, and $Y$ into two sets $Y_1$ and $Y_2$ such that $|X_1|=|Y_1| \ll |X_2|=|Y_2|$.  The edges in $X_1\cup Y_2$ and $X_2\cup Y_1$ are coloured with colour $r$.  The graph $X_1\cup Y_1$ is coloured with colours $1\dots r-1$ to produce any colouring which cannot covered by a single monochromatic path.  The graph $X_2\cup Y_2$ is coloured with colours $1\dots r-1$ to produce a $r-1$-coloured balanced complete graph which cannot be partitioned into $2r-4$ monochromatic paths after the removal of any set of at most $|X_1|$ vertices from each class of the bipartition.  The fact that such a colouring exists can be shown by induction.  

%Most existing theorems about partitioning coloured graphs partition a graph into monochromatic subgraphs which all have the same structure.  Lemma~\ref{path-bipartitecoloured} stands out from these since the two subgraphs into which it partitions $K_n$ are very different from each other.  It would be interesting to see if there are any other natural results along similar lines.   

\bigskip\noindent
\textbf{Acknowledgment}

\smallskip\noindent
The author would like to thank his supervisors Jan van den Heuvel and Jozef Skokan for their advice and discussions.

\bibliography{pathpartition}

\begin{thebibliography}{10}

\bibitem{Allen}
P.~Allen.
\newblock Covering two-edge-coloured complete graphs with two disjoint
  monochromatic cycles.
\newblock {\em Combin. Probab. Comput.}, 17(4):471--486, 2008.

\bibitem{Ayel}
J.~Ayel.
\newblock Sur l'existence de deux cycles suppl\'ementaires unicolores,
  disjoints et de couleurs diff\'erentes dans un graphe complet bicolore.
\newblock {\em Th\'ese de l'universit\'e de Grenoble}, 1979.

\bibitem{Balogh}
J.~Balogh, J.~Bar\'at, D.~Gerbner, A.~Gy\'arf\'as, and G.~S\'ark\"ozy.
\newblock Partitioning edge-2-colored graphs by monochromatic paths and cycles.
\newblock {\em preprint}, 2012.

\bibitem{Thomasse}
S.~Bessy and S.~Thomass\'e.
\newblock Partitioning a graph into a cycle and an anticycle, a proof of
  {L}ehel's conjecture.
\newblock {\em J. Combin. Theory Ser. B}, 100(2):176--180, 2010.

\bibitem{Diestel}
R.~Diestel.
\newblock {\em Graph Theory}.
\newblock Springer-Verlag, 2000.

\bibitem{Erdos}
P.~Erd\H{o}s, A.~Gy\'arf\'as, and L.~Pyber.
\newblock Vertex coverings by monochromatic cycles and trees.
\newblock {\em J. Combin. Theory Ser. B}, 51(1):90--95, 1991.

\bibitem{Gerencser}
L.~Gerencs\'er and A.~Gy\'arf\'as.
\newblock On {R}amsey-type problems.
\newblock {\em Ann. Univ. Sci. Budapest. E\"otv\"os Sect. Math}, 10:167--170,
  1967.

\bibitem{Gyarfas2}
A.~Gy\'arf\'as.
\newblock Vertex coverings by monochromatic paths and cycles.
\newblock {\em J. Graph Theory}, 7:131--135, 1983.

\bibitem{Gyarfas}
A.~Gy\'arf\'as.
\newblock Covering complete graphs by monochromatic paths.
\newblock In {\em Irregularities of Partitions, Algorithms and Combinatorics},
  volume~8, pages 89--91. Springer-Verlag, 1989.

\bibitem{Lehel}
A.~Gy\'arf\'as and J.~Lehel.
\newblock A {R}amsey-type problem in directed and bipartite graphs.
\newblock {\em Pereodica Math. Hung.}, 3:299--304, 1973.

\bibitem{Szemeredi3}
A.~Gy\'arf\'as, M.~Ruszink\'o, G.~S\'ark\"ozy, and E.~Szemer\'edi.
\newblock An improved bound for the monochromatic cycle partition number.
\newblock {\em J. Combin. Theory Ser. B}, 96(6):855--873, 2006.

\bibitem{Szemeredi1}
A.~Gy\'arf\'as, M.~Ruszink\'o, G.~S\'ark\"ozy, and E.~Szemer\'edi.
\newblock Partitioning 3-colored complete graphs into three monochromatic
  cycles.
\newblock {\em Electron. J. Combin.}, 18(1), 2011.

\bibitem{Haxell}
P.~Haxell.
\newblock Partitioning complete bipartite graphs by monochromatic cycle.
\newblock {\em J. Combin. Theory Ser. B}, 69:210--218, 1997.

\bibitem{Szemeredi2}
T.~\L{uczak}, V.~R\"odl, and E.~Szemer\'edi.
\newblock Partitioning two-colored complete graphs into two monochromatic
  cycles.
\newblock {\em Combin. Probab. Comput.}, 7:423--436, 1998.

\bibitem{Sarkozy}
G.~S\'ark\"ozy.
\newblock Monochromatic cycle partitions of edge-colored graphs.
\newblock {\em J. Graph Theory}, 66:57--64, 2011.

\end{thebibliography}
\bibliographystyle{abbrv}
\end{document}